\documentclass[12pt,leqno,fleqn]{amsart}  
\usepackage{amsmath,amstext,amsthm,amssymb,amsxtra}
\usepackage[top=1.5in, bottom=1.5in, left=1.25in, right=1.25in]	{geometry}
\usepackage{txfonts} 
\usepackage[T1]{fontenc}
\usepackage{lmodern}

 \usepackage{euler}   

\usepackage{tikz}

\usepackage{mathtools}
\mathtoolsset{showonlyrefs,showmanualtags}

\usepackage{hyperref} 
\hypersetup{
    colorlinks=true,       
    linkcolor=blue,          
    citecolor=magenta,        
    filecolor=magenta,      
    urlcolor=cyan           
}

\usepackage[msc-links]{amsrefs}

\theoremstyle{plain} 
\newtheorem{lemma}[equation]{Lemma} 
\newtheorem{proposition}[equation]{Proposition} 
\newtheorem{theorem}[equation]{Theorem} 
 
\newtheorem{conjecture}[equation]{Conjecture}

\newtheorem{priorResults}{Theorem}

\theoremstyle{definition}

\theoremstyle{remark}

\newtheorem*{ack}{Acknowledgment}

\numberwithin{equation}{section}

\title[] {Sparse Bounds for Bochner-Riesz Multipliers}
\author[M. T. Lacey] {Michael T. Lacey}   

\address{ School of Mathematics, Georgia Institute of Technology, Atlanta GA 30332, USA}
\email {lacey@math.gatech.edu}
\thanks{M.L.: Research supported in part by grant   National Science Foundation grant DMS-1600693, and by Australian Research Council grant DP160100153. All Authors:  This material is based upon work supported by the National Science Foundation under Grant No. DMS-1440140, while the authors were in residence at the Mathematical Sciences Research Institute in Berkeley, California, during the Spring 2017 Semester. }

 \author[D. Mena]{Dar\'io Mena Arias}   

\address{ School of Mathematics, Georgia Institute of Technology, Atlanta GA 30332, USA}
\email {dario.mena@math.gatech.edu}

\author[M. C. Reguera]{Maria Carmen Reguera}

\address{School of Mathematics, 
University of Birmingham,  
Edgbaston,  
Birmingham UK,  
B15 2TT}

\begin{document}

\begin{abstract}
The Bochner-Riesz multipliers $ B _{\delta }$ on $ \mathbb R ^{n}$ are shown to satisfy a range of sparse bounds,
for all $0< \delta < \frac {n-1}2 $.  
The range of sparse bounds increases to the optimal range, 
as $ \delta $ increases to the critical value,  $ \delta =\frac {n-1}2$, even assuming only partial information on the Bochner-Riesz conjecture in dimensions $ n \geq 3$.  In dimension $n=2$, we prove a sharp range of sparse bounds.  
The method of proof is based upon a `single scale' analysis, and yields the sharpest known weighted estimates for the Bochner-Riesz multipliers in the category of Muckenhoupt weights. 
\end{abstract}
	\maketitle  
\section{Introduction} 

We study \emph{sparse bounds} for the Bochner-Riesz multipliers in dimensions $ n \geq 2$. 
The latter are Fourier multipliers   $ B _{\delta } $  with symbol $ (1 - \lvert  \xi \rvert^2 )_+ ^{\delta }$, for $ \delta >0$.  That is, 
\begin{equation*}
\mathcal F  B _{\delta }f =  (1 - \lvert  \xi \rvert^2 )_+ ^{\delta } \mathcal F f , 
\end{equation*}
where $ \mathcal F$ is a choice of Fourier transform.  
At $ \delta_n = \frac{n-1}2$, the multiplier is borderline Calder\'on-Zygmund, and  one has the very sharp bounds of Conde-Alonso, Culiuc, di Plinio and Ou \cite{2016arXiv161209201C}, which we recall in Theorem~\ref{t:critical} below. 
In this paper, we focus on the super-critical range $ 0< \delta < \frac{n-1}2$, the study of which was initiated by 
Benea, Bernicot and Luque  \cite{160506401}.  
We supply sparse bound for all $ 0< \delta < \frac{n-1}2$, and  prove a sharp range of estimates in dimension $n=2$.  

\smallskip 
Sparse bounds are a particular quantification of the (weak) $ L ^{p}$-bounds for an operator, which in particular immediately imply weighted and vector-valued inequalities.  
The topic has been quite active, with an especially relevant paper being that of Benea, Bernicot and Luque 
\cite{160506401}, but also see \cites{160305317,MR3531367,2017arXiv170204569H,2016arXiv161209201C,161001531,2017arXiv170208594L} for more information about this topic.   
We set notation for the sparse bounds.  
Call a collection of cubes $ \mathcal S$  in $ \mathbb R ^{n}$ \emph{sparse} if there 
are sets $ \{ E_S  \,:\, S\in \mathcal S\}$  
which are pairwise disjoint,   $E_S\subset  S$ and satisfy $ \lvert  E_S\rvert > \tfrac 14 \lvert  S\rvert  $ for all $ S\in \mathcal S$.
For any cube $ Q$ and $ 1\leq r < \infty $, set $ \langle f \rangle_ {Q,r} ^{r} = \lvert  Q\rvert ^{-1} \int _{Q} \lvert  f\rvert ^{r} dx  $.  Then the $ (r,s)$-sparse form $ \Lambda _{\mathcal S, r,s} = \Lambda _{r,s} $, indexed by the sparse collection $ \mathcal S$ is 
\begin{equation} \label{e:sparse_def}
\Lambda _{S, r, s} (f,g) = \sum_{S\in \mathcal S} \lvert  S\rvert \langle f  \rangle _{S,r} \langle g \rangle _{S,s}.  
\end{equation}

Given a  sublinear operator $ T$, and $ 1\leq r, s < \infty$, we set 
$ \lVert T \,:\, (r,s)\rVert$ to be the infimum over constants $ C$ so that for all bounded compactly supported functions $ f, g$, 
\begin{equation}\label{e:SF}
\lvert  \langle T f, g \rangle \rvert \leq C \sup  \Lambda _{r,s} (f,g), 
\end{equation}
where the supremum is over all sparse forms.  
It is essential that the sparse form be allowed to depend upon $ f $ and $ g$. But the point is that the sparse form itself varies over a class of operators with very nice properties.

The study of sparse bounds for the Bochner-Riesz multipliers was initiated by Benea, Bernicot and Luque \cite{160506401}, who established sparse bounds for a restricted range of parameters $\delta, r  $ and $ s$ below.  
We extend their results, using an alternate, less complicated method of proof, yielding results for all $ \delta >0$.  
In two dimensions our main result is as follows.    

\begin{theorem}\label{t:2} Let $ n=2$, and $ 0< \delta < \tfrac 12 $. 
Let $ \mathbf R (2, \delta )$ be the  open trapezoid  with vertices 
\begin{gather*}
v _{2, \delta ,1} = ( \tfrac {1-2 \delta }4, \tfrac {3 +2 \delta }4 ), \quad 
v _{2, \delta ,2} = ( \tfrac {1+6 \delta }4 , \tfrac {3 +2 \delta }4  ), \quad 
\\
v _{2, \delta ,3} = ( \tfrac {3 +2 \delta }4  ,  \tfrac {1+6 \delta }4 ), \quad 
v _{2, \delta ,4} = ( \tfrac {3 +2 \delta }4,  \tfrac {1-2 \delta }4). 
\end{gather*}
(See Figure~\ref{f:2}.)
There holds 
\begin{equation}\label{e:2sparse}
\lVert B _{\delta } :  (r,s)\rVert < \infty , \qquad  (\tfrac 1 r, \tfrac 1 s) \in   \mathbf R (2, \delta ). 
\end{equation}
Moreover, the inequality above fails  for $\frac 1r + \frac 1s >1$, with $(\frac 1r, \frac 1s)$ not in the closure of $\mathbf{R}(2, \delta)$. 
\end{theorem}

As $ \delta $ increases to the critical value of $ \delta = \frac{1}2$, the trapezoid $ \mathbf R (2, \delta )$ increases to the upper triangle  with vertices $ (1,0)$, $ (0,1)$ and $ (1,1)$. This is the full arrange allowed for the case of $ \delta = \frac{1}2$, see Theorem~\ref{t:critical}.

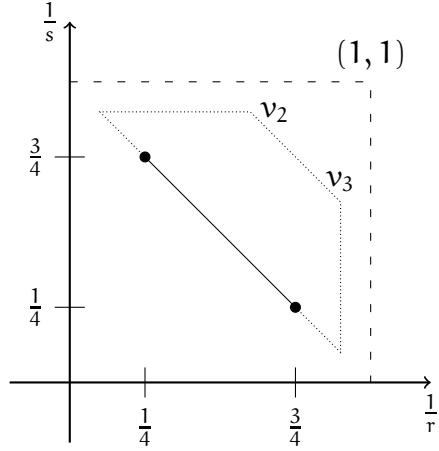
\begin{figure}
\begin{tikzpicture}[scale=4] 
\draw[thick,->] (-.2,0) -- (1.2,0) node[below] {$ \frac 1 r$};
\draw[thick,->] (0,-.2) -- (0,1.2) node[left] {$ \frac 1 s$};
\draw[densely dotted] (.75,.25) --  (.9,.1)
-- (.9,.6) node[above] {$ v_3$} 
-- (.6,.9) node[right] {$ v_2$} 
--  (.1,.9) 
-- (.25,.75) ;
\filldraw  (.25,.75) circle (.04em); \filldraw  (.75,.25) circle (.04em); 
\draw  (.25,.75) -- (.75,.25); 
\draw (.25,.05) -- (.25,-.05) node[below] {$ \tfrac 14$};
\draw (.75,.05) -- (.75,-.05) node[below] {$ \tfrac 34$};
\draw (.05,.25) -- (-.05,.25) node[left] {$ \tfrac 14$};
\draw (.05, .75) -- (-.05, .75) node[left] {$ \tfrac 34$};
\draw[loosely dashed] (0,1) -- (1.,1.) node[above] {$ (1,1)$} -- (1.,0); 
\end{tikzpicture}

\caption{The trapezoid $ \mathbf R (2, \delta )$ of Theorem~\ref{t:2}.  The Bochner-Riesz bounds are sharp at the indices $ \frac{1}p =\frac 34, \frac{1} 4$, which corresponds to the Carleson-Sj\"olin bounds. The sparse bounds for $ B _{\delta }$ hold for all $ (\frac 1 r,\frac 1 s)$ inside the dotted trapezoid, and fails outside the trapezoid. We abbreviate $ v _{2, \delta ,2}=v_2$, and similarly for $ v_3$. } 
\label{f:2}
\end{figure}

In the next section, we give the full statement of the results in all dimensions.
  The remainder of the paper is taken up recalling some details about sparse bounds, the (short) proof of the main results, and then drawing out the weighted corollaries.

\section{The Full Statement} 

In dimensions $ 3$ and higher, we only have partial  information about the Bochner-Riesz conjecture.  
Nevertheless, we show that from this partial information one can obtain sparse bounds as a consequence. 

\begin{conjecture}\label{j:BR}[Bochner-Riesz Conjecture]   We have $ B _{\delta } : L ^{p} (\mathbb R ^{n}) \mapsto L ^{p} (\mathbb R ^{n})$ if 
\begin{equation} \label{e:BR}
n\bigl\lvert \tfrac 1p - \tfrac 12\bigr\rvert < \tfrac 12 + \delta , \quad \delta >0. 
\end{equation}

\end{conjecture}

This has an  equivalent formulation, in terms of `thin annuli,' which is the form we prefer.   
Let $  \mathbf 1_{ [-1/4, 1/4]} \leq \chi   \leq  \mathbf 1_{[-1/2,1/2]}$ be a Schwartz function, and set 
$ S _{\tau  }$ to be the Fourier multiplier with symbol $ \chi (  (\lvert  \xi \rvert -1)/ \tau  )$. 

\begin{conjecture}\label{j:S}  Subject to the condition  
$  n\bigl\lvert \tfrac 1p - \tfrac 12\bigr\rvert < \tfrac 12 $, there holds 
\begin{equation} \label{e:S}
    \lVert S _{\tau  }\rVert _{L ^{p} \mapsto L ^{p}}   \lesssim _{\epsilon } 1,  
\qquad  0 < \tau  < 1.  
\end{equation}
\end{conjecture}

Above, we use the notation $ A (\tau ) \lesssim _{\epsilon } B (\tau )$ to mean that 
for all $ 0< \epsilon < 1$, there is a constant $ C_ \epsilon $ so that uniformly in $ 0< \tau  <1$, there holds 
$ A (\tau ) \leq C _{\epsilon } \tau ^{- \epsilon } B (\tau )$.  
It is typical in these types of questions that one expects losses in $ \tau $ that are of a logarithmic nature at end points. This issue  need not concern us. 

The Theorem below takes as input partial information about the Bochner-Riesz Conjecture and deduces 
a range of sparse bounds.  
For $ 0< \delta < \frac {n-1}2$, 
let $ \mathbf R (n, p_0, \delta )$ be the open trapezoid with vertexes 
\begin{gather} \label{e:v}
v _{n, \delta ,1} =  ( \tfrac 1 {p_0} (1- \tfrac {2 \delta } {n-1}) , \tfrac 1 {p_0'} +\tfrac 1 {p_0}\tfrac {2 \delta } {n-1}), 
\qquad 
v _{n,\delta ,2} = (\tfrac 1 {p_0}+\tfrac 1 {p_0'} \tfrac {2 \delta } {n-1}, \tfrac 1 {p_0'} +\tfrac 1 {p_0}\tfrac {2 \delta } {n-1}) ,
\\
v _{n,\delta ,3} = \overline {v _{n, p_0,2}} , 
\quad 
v _{n, \delta ,4} = \overline {v _{n, p_0,1}},  \quad \textup{where $ \overline {(a,b)}= (b,a)$.} 
\end{gather}

\begin{theorem}\label{t:higher}  Assume dimension $ n\geq 2$.  And let $ 1< p_0 < 2 $ be such that  the estimate \eqref{e:S} holds.  
  Then, for $ 0< \delta < \frac {n-1}2$, the following sparse bound hold. 
\begin{equation}\label{e:higher}
\lVert B _{\delta } :  (r,s)\rVert < \infty , \qquad  (\tfrac 1 r, \tfrac 1 s) \in \mathbf R (n, p_0, \delta ). 
\end{equation}
Moreover, 
for the critical value of $p = p (\delta)$ given by  $ \tfrac n {p_\delta} =  \tfrac {n+1} 2 + \delta$,  
the inequality above fails  for $\frac 1r + \frac 1s >1$, with $(\frac 1r, \frac 1s)$ not in the closure of $\mathbf{R}(n,  p_\delta, \delta)$.
\end{theorem}

This Theorem contains Theorem~\ref{t:2}, since the Bochner-Riesz conjecture holds in dimension $ 2$, 
as was proved by Carleson-Sj\"olin \cite{MR0361607}.  (Also see C\'ordoba \cite{MR544242}.)  In dimensions $ n\ge3$, the best results are currently due to Bourgain-Guth,  \cite{MR2860188}, but also see Sanghyuk Lee \cite{MR2046812}. 
We summarize the best known information in 

\begin{priorResults}\label{t:BR} These positive results hold for the Bochner-Riesz Conjecture. 
\begin{enumerate}
\item  \cite{MR0361607} In the case of $ n=2$, \eqref{e:BR} holds.  

\item \cite{MR2860188}*{Thm 5} In the case of $ n\geq 3$, the condition \eqref{e:BR} holds if 
$q=\max (p,p')$ satisfies  
\begin{equation}
q >\begin{cases}
 2 \frac {4n+3}{4n-3} & n\equiv 0 \mod 3 
\\
  \frac {2n+1}{n-1} & n\equiv 1 \mod 3
\\
  \frac {4(n+1)}{2n-1} & n\equiv 2 \mod 3 
\end{cases}
\end{equation}
\end{enumerate}
\end{priorResults}

Concerning sparse bounds for the Bochner-Riesz multipliers, the general result of Benea, Bernicot and Luque \cite{160506401} is a bit technical to state in full generality. We summarize it as follows. 

\begin{priorResults}\label{t:BRS} These two results hold. 
\begin{enumerate}
\item \cite{160506401}*{Thm 1} In dimension $ n=2$, for $ \delta > \tfrac 16$, we have $ \lVert B _{\delta } : (\tfrac 65,2)\rVert < \infty $.  

\item \cite{160506401}*{Thm 3} In dimensions $ n>3$,  for all $ \delta >0$, there is a $ 1< p (\delta )< 2$ for which we have 
$ \lVert B _{\delta } : (p (\delta ),2)\rVert < \infty $. 
(Using our notation,  the sparse bound holds when the second coordinate of $ v _{n, \delta ,2}$ is $ \frac 12$.) 
\end{enumerate}

\end{priorResults}

Our result provides sparse bounds for the Bochner-Riesz multipliers, for all $ \delta >0$, and all $ p$ in a non-trivial interval around $ 2$.  
It is a routine exercise to verify that a consequence of Theorem~\ref{t:2} that we have 
\begin{equation}\label{e:56}
 \lVert B _{\delta } : (2, \tfrac 65)\rVert < \infty, \qquad n=2,\  \delta >\tfrac 16. 
\end{equation}
Indeed, using the notation Theorem~\ref{t:2}, we have $ v _{2,  1/6 ,2} = (\tfrac 12, \tfrac 56)$.  
This is  the two dimensional result of \cite{160506401}.  

The interest in sparse bounds, besides their quantification of $ L ^{p}$ bounds, is that 
they quickly deliver weighted and vector-valued inequalities.  In many examples, these estimates are sharp 
\cites {MR3531367,MR3625108,2016arXiv161209201C,160401334}, or dramatically simplify existing proofs, and provide weighted inequalities in settings where none were known before \cites{160906364,160908701,2017arXiv170208594L}.
The mechanism to do this is already well represented in the literature \cite{MR3531367}, and was initiated by Benea, Bernicot and Luque \cite{160506401} in the setting of Bochner-Riesz multipliers. 
We point the interested reader there for more information about weighted estimates in the Bochner-Riesz setting. 

\smallskip 

That our result and that of \cite{160506401} coincide at the case of $ r=2$ is not so surprising.  
They approach the problem by using sharp results about spherical restriction, as there is a close connection between the Bochner-Riesz Conjecture and spherical restriction, subject to an index in the restriction question being 2. 
Our approach is more direct, working essentially with the 'single scale' version of the Bochner-Riesz Conjecture directly, through Conjecture~\ref{j:S}.  In both cases, we use the `optimal' unweighted estimates, and derive the sparse bounds.  

Concerning the critical index $ \delta_n = \frac {n-1}2$, 
it is well known that the Bochner-Riesz operator is borderline Calder\'on-Zygmund.  Hence one expects much better sparse bounds.  
The best sharp bound is due to Conde-Alonso, Culiuc, di Plinio and Ou \cite{2016arXiv161209201C}.  
It shows not only sparse bounds in the upper triangle of the $ (\frac 1 r,\frac 1 s)$ plane, but also a quantitative estimate at the vertex $ (1,1)$.  

\begin{priorResults}\label{t:critical} \cite{2016arXiv161209201C}  
In all dimensions $ n\geq 2$, we have 
\begin{equation*}
\lVert B _{\delta_n } : (1, 1+ \epsilon )\rVert \lesssim \epsilon ^{-1}, \qquad 0< \epsilon < \infty .  
\end{equation*}

\end{priorResults}

Note that the trapezoid of our theorem increases to the upper triangle, as $ \delta $ increases to the critical index $ \delta_n = \frac {n-1}2$.  In that sense, our results `interpolates' the better bounds known in the critical case.
We do not recover Theorem \ref{t:critical}.  Indeed we can't as the proof is intrinsically multiscale, whereas ours is not.

\medskip 

The sparse bounds imply vector-valued and weighted inequalities for the Bochner-Riesz multipliers. 
The weights allowed are in the intersection of certain Muckenhoupt and reverse H\"older classes. 
The inequalities we can deduce are strongest at the vertex $ v _{n, \delta ,2}$, using the notation of \eqref{e:v}. 
Indeed, the weighted consequence is the strongest known for the Bochner-Riesz multipliers. 
The method of deduction follows the model of arguments in \cite{160506401}* {\S 7} and  \cite {MR3531367}*{\S 6}, and so we suppress the details.  

\bigskip 

We conclude with these remarks.       
\begin{enumerate}
\item  Seeger \cite{MR1405600} proves an endpoint weak-type result for the Bochner-Riesz operators in the plane.  The sparse refinement of that is given Kesler  and one of us in \cite{170705844}. 

\item  Extensions of these results to maximal Bochner-Riesz operators is hardly straight forward.  For relevant norm inequalities, see \cites{MR768732,MR1666558,MR2046812}. 

\item Bak \cite{MR1371114} proves endpoint estimates for negative index Bochner-Riesz multipliers. (Also see  Guti\'errez \cite{MR1641626}.)  Aside from endpoint issues, it would be easy to derive sparse bounds for these operators using the techniques of this paper. The $A_{p,q}$ weighted consequences would be new, it seems. The endpoint issues would be interesting.  

\item  It is also of interest to obtain weighted bounds that more explicitly involve the Kakeya maximal function, as is done by Carbery \cite {MR768732} and Carbery and Seeger \cite{MR1765787}. This would require substantially new techniques. 
\end{enumerate}

\begin{ack}
We benefited from conversations with   Andreas Seeger and Richard Oberlin, as well as  careful readings by referees.  
\end{ack}

\section{Background on Sparse Forms} 

We collect some facts concerning sparse bounds.  It is a useful fact that given bounded and compactly supported functions, there is basically one form that controls all others. 

\begin{proposition}\label{p:One} \cite{161001531}* {\S 4} Given $ 1\leq r , s< \infty $, and bounded and compactly supported functions $ f,$ and $ g$, there is a single sparse form $ \Lambda _{\mathcal S_0 ,r ,s } $ for which 
\begin{equation*}
\sup _{\mathcal{S}} \Lambda _{\mathcal{S},r,s} (f,g) \lesssim \Lambda _{\mathcal S_0,r,s} (f,g).   
\end{equation*}
The implied constant is only a function of dimension.  
\end{proposition}

Second, closely related sparse forms are also controlled by the sparse forms we defined at the beginning of the paper.  For a cube $ Q$, and $ 1 \leq r  < \infty $, set a non-local average to be 
\begin{equation} \label{e:ll}
\langle \! \langle  f \rangle \! \rangle _{Q,r} = \Bigl[ \lvert  Q\rvert ^{-1}    \int  \lvert  f (x)\rvert ^{r}[ 1+ \textup{dist} (x,Q) / |Q|] ^{-(n+1)}dx \Bigr] ^{\frac 1 r}. 
\end{equation} 
And then define a sparse form $ \  \Lambda ' _{\mathcal S, r,s}$ using the non-local averages above in place of $ \langle f \rangle _{Q,r}$. 
These forms are not essentially larger. 

\begin{proposition}\label{p:llgg} \cite{2016arXiv161208881C}*{Lemma 2.8} For bounded and compactly supported functions $ f,g$, and $ 1\leq r,s< \infty $, 
we have 
\begin{equation*}
\sup _{\mathcal S}   \Lambda ' _{\mathcal S, r,s} (f,g) \lesssim 
\sup _{\mathcal S}   \Lambda _{\mathcal S, r,s} (f,g). 
\end{equation*}
\end{proposition}

A central point is that the selection of the 'optimal' sparse form in Proposition~\ref{p:One} is certainly non-linear. But at the same time, one would ideally like to interpolate sparse bounds. We do not know how to do this in general, but the analysis of the operators $ S _{\tau }$, being `single scale', places us in a situation where we can interpolate.  

Using the notation of \eqref{e:ll}, for $ 0< \tau  <1$, set  
\begin{equation}\label{e:til}
\tilde \Lambda _{\tau , r,s} (f,g) = 
\sum_{\substack{Q\in \mathcal D\\  1\leq \ell Q \leq \frac1{ \tau ^{1+ \eta }}  }}
\!\!\langle \! \langle  f \rangle \! \rangle_{Q,r} \langle \! \langle  g \rangle \! \rangle _{Q,s}\lvert  Q\rvert, \qquad 0< \eta < 1. 
\end{equation}
Above, $ \mathcal D$ denotes the dyadic cubes in $ \mathbb R ^{n}$, and $ \ell Q = \lvert  Q\rvert ^{\frac 1n} $ is the side length of $ Q$. That is, the sum is over all dyadic subcubes with side length between 1 and $ 1/ \tau ^{1+ \eta}  $.  
We have  this interpolation fact. 

\begin{proposition}\label{p:interpolate} Let $ 1\leq r_j,s_j \leq \infty $ for $ j=0,1$ and fix $0< \tau <  \infty $. 
Suppose that for some linear operator $ T$ we have 
\begin{equation}\label{e:01}
\lvert  \langle T f, g \rangle\rvert  \leq C_j \tilde \Lambda _{\tau , r_j,s_j} (f,g), \qquad j=0,1, 
\end{equation}
for all smooth compactly supported functions $f,g$.  
Then, for $ 0< \theta < 1$, we have 
\begin{equation}\label{e:theta}
\lvert  \langle T f, g \rangle\rvert  \leq C_0 ^{\theta } C_1 ^{1- \theta } \tilde \Lambda _{\tau , r_ \theta ,s_ \theta} (f,g), 
\end{equation}
where $\frac{1}{r_\theta} = \frac{\theta}{r_0}+\frac{1-\theta}{r_1}$, and similarly for $s _{\theta}$.  
\end{proposition}

The proof is a variant of  Riesz-Thorin interpolation, but we include some details, as the proposition is new in this context, as far as we know.
\begin{proof}
Let us recast the sparse bound in a slightly more general format.  
For cubes $ Q$, set 
\begin{equation*}
\langle \! \langle  f \rangle \! \rangle_{\lambda , Q,r} = \Bigl[ \frac{1}{|Q|} \int  \frac{ \lvert  f (x)\rvert ^{r} } { (1+ \operatorname {dist} (x,Q) / |Q| ) ^{n+1}}  \   {d \lambda (x)}  \Bigr] ^{\frac 1 r}. 
\end{equation*}
Above, $ \lambda $ is some Borel measure.  Fix a finite collection of cubes $ \mathcal Q$, and consider 
a `sparse form' given by 
\begin{equation} \label{e:B}
B _{r,s} (f,g) = B _{\mathcal Q, \lambda , w,r,s } (f,g)= \sum_{Q\in \mathcal Q} w (Q) \{f\} _{\lambda , Q,r} \{g\} _{\lambda ,Q,s}. 
\end{equation}
Above $ w : \mathcal Q \mapsto (0, \infty )$ is a non-negative function.  The sparse forms that we consider are special instances of these more general forms. 

Appeal to H\"older's inequality.  Given $ r_0 < r_1$ and  $ s_0 < s_1$, 
we have 
\begin{equation*}
B _{ r _{\theta },s _{\theta } } (f,g) 
\leq B _{ r_0,s_0 } (f,g) ^{\theta  }B _{ r_1,s_1 } (f,g) ^{1-\theta  }, 
\qquad 0< \theta < 1
\end{equation*}
where $ \frac{1} {r_ \theta } = \frac \theta {r_0} + \frac{1- \theta } {r_1}$, and similarly for $ s _{\theta }$. 

\smallskip
Let us consider a bilinear form $ \beta $ for which we have 
\begin{equation*}
\lvert  \beta (f,g)\rvert \leq A_j  B _{ r_j,s_j } (f,g), \qquad j=0,1.  
\end{equation*}
By multiplying $ f$, $g$ and the  measure $ \lambda  $ by various constants, we can assume that 
\begin{equation*}
 B _{ r_j,s_j } (f,g)=1 , \qquad  j=0,1. 
\end{equation*}  
For $ 0< \theta < 1$, consider the holomorphic function $ F (s) = 
\beta (f_s, g_s)$, where 
\begin{equation*}
f_s = \textup{sgn} (f) \lvert  f\rvert ^{ (1-s) \frac {r_ \theta } {r_0} + s  \frac{r_ \theta } {r_1} } ,  
\end{equation*}
and similarly for $ g_s$.   The function $ F (s)$ is of at most exponential growth in the strip $ 0 \le \operatorname  {Re} s \leq 1$. Namely, 
\begin{equation*}
\lvert F(s)\rvert 
\leq B_{r_1,s_1} (f_s,g_s) \leq C e^{C \lvert s\rvert } ,  \qquad 0<  \operatorname  {Re} s \leq 1. 
\end{equation*}
for some finite positive constant $C$. This is because  $f$ and $g$ are bounded functions,  and we have a  finite collection of cubes $\mathcal Q$.  Our deduced bounds are independent of these \emph{a priori} assumptions. 
It also holds that  $ \lvert  F (j + i \sigma )\rvert \leq  A_j  $, for $ j=0,1$.  It follows from Lindel\"of's Theorem that $ F$ is log-convex on $ [0,1]$. In particular, 
\begin{equation*}
\lvert  F (\theta ) \rvert = \lvert  \beta (f,g)\rvert \leq A_0 ^{\theta } A_1 ^{1- \theta }  . 
\end{equation*}
From this, we conclude our proposition. 
\end{proof}

\section{Proof of the Sparse Bounds} 

The connection between the Bochner-Riesz  and the $ S _{\tau }$ multipliers is 
well-known, and central to standard papers in the subject like \cites{MR544242,MR768732}. We briefly recall it here. 
For each $ 0 < \delta < \frac {n-1}2$, we have 
\begin{equation*}
B _{\delta } = T_0 + \sum_{k=1} ^{\infty } 2 ^{- k \delta } \operatorname {Dil} _{1- 2 ^{-k}} S_ {2 ^{-k}}, 
\end{equation*}
where these conditions hold:  First, $ T_0$ is a Fourier multiplier, with the multiplier being a Schwartz function supported near the origin.  The operator $ \operatorname {Dil}_s f (x) = f (x/s)$ is a dilation operator. 
And, each $ S _{2 ^{-k}}$ is a Fourier multiplier with symbol $ \chi _k ( 2 ^{k}\bigl\lvert \lvert  \xi \rvert -1\bigr\rvert)$, where the $ \chi _k $ satisfy a uniform class of derivative estimates.  

The point is then to show this result, in which we exploit the openness of the condition we are seeking to prove.  

\begin{theorem}\label{t:S}
Assume dimension $ n\geq 2$.  And let $ 1< p_0 < 2 $ be such that  the estimate \eqref{e:S} holds.  
Then, the following sparse bounds hold. 
For all  $ (\frac 1 r, \frac 1 s) \in \mathbf R (n, p_0, \delta )$,  
there is a $ \kappa = \kappa (r,s) >0 $ so that 
\begin{equation}\label{e:higherS} 
\lVert  S _{\tau }  :  (r,s)\rVert  \lesssim _{\epsilon } \tau ^{- \delta + \kappa } , \qquad 0< \tau < 1. 
\end{equation}
\end{theorem}

The papers of C\'ordoba \cites{MR544242,MR0447949} also proceeds by analysis of the operators $S_\tau$. Also see Duoandikoetxea \cite{MR1800316}*{Chap 8.5}. 
Write   $ S _{\tau } f  = K _{\tau } \ast f $.    The basic properties of this operator and kernel that we need are these.  

\begin{lemma}\label{l:S} For $ 0< \tau < \frac{1}2$, these properties hold. 
\begin{enumerate}

\item We have this estimate for the kernel $ K _{\tau }$.  For all $ 0 < \eta < 1$ and $ N > 1$, 
\begin{equation}\label{e:K}
\lvert K _{\tau } (x) \rvert   
\lesssim \tau  \cdot  \begin{cases}
[ 1 + \lvert  x\rvert ] ^{-\frac {n-1}2  }   & \lvert  x\rvert < C \tau ^{-1 - \eta } 
\\
\lvert x\rvert ^{\frac{1-n}2} [ \tau \lvert  x\rvert ]^{-N}  & \textup{otherwise}.   
\end{cases}
\end{equation}
The implied constants depend upon $0< \eta <1$, and $N>1$.  

\item   $ \lVert S _{\tau }\rVert _{1 \mapsto 1} \lesssim \tau ^{- \frac{n-1}2}$. 

\end{enumerate}

\end{lemma}

\begin{proof}

The second estimate follows from the first.  The first is seen this way.  
Let $ \sigma $ denote normalized Haar measure on the sphere $ S ^{n-1} \subset \mathbb R  ^{n}$.  
Then, recall \cite{MR1232192}*{Chap VIII.3}, that the Fourier transform of $ \sigma $ has an expression in terms of Bessel functions as follows.  
\begin{align*}
\widehat {d \sigma } (x) = \int _{S ^{n-1}} e ^{- 2 \pi i x \cdot \xi } d \sigma (\xi ) 
= 2 \pi \lvert  x \rvert ^{\frac{2-n}2} J _{\frac{n-2}2} (2 \pi \lvert  x\rvert ).  
\end{align*}
The Bessel function has an expansion  \cite{MR1232192}*{Chap VIII.1.4} 
\begin{equation*}
J _{\frac{n-2}2} ( s ) = \sqrt {\tfrac{2 } {\pi s}} \cos \bigl(s -  \pi \tfrac{n-3}4  \bigr) + O (s ^{-3/2}), \qquad s\to \infty . 
\end{equation*}
It is preferable to write 
\begin{equation}\label{e:J}
J _{\frac{n-2}2} ( s ) =  \sqrt {\tfrac1{s}}[ e ^{+i s} a_+ (s) +  e ^{- is} a _{-} (s)],  \qquad s>0
\end{equation}
where 
\begin{equation} \label{e:a}
\bigl\lvert  \frac{d ^{m}} {ds ^{m}} a _{\pm} (s) \bigr\rvert  \lesssim [ 1+ s ] ^{-m}, \qquad m\in \mathbb N , \ s>0. 
\end{equation}
This follows from the asymptotic expansion of the Bessel functions.

Turning to our estimates for  $ K _{\tau } (x)$,   it is clear that $ \lvert  K _{\tau } ( x)\rvert 
\leq \lVert  \chi (\tau ^{-1} \lvert  \lvert  x\rvert -1\rvert) \rVert_1 \lesssim \tau $, since we are only estimating the volume of a thin annulus. Thus, we only need to consider $ \lvert  x\rvert \gtrsim 1 $ below.  
In terms of the Bessel function, we have 
\begin{equation} 
\begin{split}
K _{\tau } (x)   & = \int \chi (\tau ^{-1} \lvert  \lvert  \xi \rvert -1\rvert) \operatorname e ^{i x \cdot \xi } \,d \xi 
\\
& = n\int _{S ^{n-1}} \int _{0} ^{2} 
\chi (\tau ^{-1} \lvert  r -1\rvert) \operatorname e ^{i rx \cdot \xi } d \sigma (\xi)\,  r ^{n-1}\,dr 
\\
& =  2 n\pi \lvert  x \rvert ^{\frac{2-n}2}
\int_0 ^{2}  \chi (\tau ^{-1} \lvert r-1\rvert) r ^{ \frac{n}2} J _{\frac {n-2}2} (2 \pi r\lvert  x\rvert )\, dr
\end{split}
\end{equation}
Using \eqref{e:J}, this last expression is the sum of two terms, both of a similar nature.   The first term is 
\begin{equation*}
2 n\pi \lvert  x \rvert ^{\frac{1-n}2} 
\int_0 ^{2}  \chi (\tau ^{-1} \lvert r-1\rvert) r ^{ \frac{n-1}2} a _{+} (r \lvert  x\rvert ) e ^{i r \lvert  x\rvert } \, dr . 
\end{equation*}
The integral above is obviously dominated by $ \tau $,  and this is the estimate that we use for $ 1 \leq \lvert  x\rvert < \tau ^{-1 - \eta } $.  
For $ \lvert  x\rvert \geq \tau ^{-1- \eta } $, we can employ the standard integration by parts argument and the derivative conditions in \eqref{e:a}.  
\end{proof}

The decay condition in \eqref{e:K}  reveal that for fixed $ \tau $, we need not be concerned with the full complexity of the sparse bound. We can rather work with this modified definition.  Recall the  sparse form $\tilde \Lambda _{\tau, r,s }$ defined in  \eqref{e:til}, where we restrict cubes to have side length at least one, and no more than $ \tau ^{-1- \eta }$.  And for which we have the interpolation estimate of Proposition~\ref{p:interpolate}.  We define 
$ \lVert T : (r,s, \tau )\rVert$ to be the best constant $ C$ in the inequality 
\begin{equation}\label{e:TIL}
\lvert  \langle T f, g \rangle\rvert  
\leq C \tilde \Lambda _{\tau, r,s } (f,g),  
\end{equation}
the inequality holding uniformly over all bounded and compactly supported functions $ f,g$.

\begin{lemma}\label{l:main} 
Assume dimension $ n\geq 2$.  And let $ 1< p_0 < 2 $ be such that  the estimate \eqref{e:S} holds.  
These sparse bounds hold, for all $  0< \tau  , \eta < 1$. 
\begin{align}\label{e:main11}
\lVert S _{\tau } : (1,1, \tau )\rVert  & \lesssim   \tau ^{- \frac {n-1}2 - n \eta }, 
\\  \label{e:main1inft}
\lVert S _{\tau } : (1, \infty, \tau  )\rVert  & \lesssim   \tau ^{- \frac {n-1}2  - n \eta  }, 
\\ \label{e:mainpc}
\lVert S _{\tau } : (p_0,p_0', \tau )\rVert  & \lesssim  \tau ^{- \eta }
\end{align}
The implied constants depend upon $0< \eta <1$. 
\end{lemma}

\begin{proof}
It is in the last condition \eqref{e:mainpc} that the hypothesis is important. Note that if $ f, g$ are supported on a cube $ Q$ of side length $ \tau ^{-1 - \eta  }$, then we have from the assumption that the Bochner-Riesz estimate \eqref{e:S} holds for $ p= p_0$, 
\begin{align*}
\lvert  \langle S _{\tau } f, g \rangle\rvert & \lesssim \tau ^{- \epsilon }
\lVert f\rVert _{p_0} \lVert g\rVert _{p_0'} \lesssim _{\epsilon } \langle f \rangle _{Q, p_0} \langle g \rangle _{Q, p_0'} \lvert  Q\rvert.   
\end{align*}
In view of the decay beyond the scale $ \tau ^{-1- \eta } $ in \eqref{e:K},  and the global form of the average in \eqref{e:til},  we can easily complete the proof of the claimed bound.  (And, we only need to use the dyadic cubes of scale $ \tau ^{-1 - \eta }  $, rather than the full range of scales in \eqref{e:til}.) 

\smallskip 
In a similar way, if $ f$ and $ g$ are supported on a cube of side length $ \tau ^{-1 - \eta }$, one can use the 
kernel decay in  \eqref{e:K} to see that 
\begin{align*}
\lvert  \langle S _{\tau } f, g \rangle\rvert &  
\leq \lVert K _{\tau } (x)  \rVert_1 \lVert f\rVert_1 \lVert g\rVert_ \infty 
\\
& \lesssim \tau ^{- \frac {n-1}2 - n \eta } \lVert f\rVert_1 \lVert g\rVert_ \infty  
\lesssim \tau ^{- \frac {n-1}2 - n \eta } \langle f \rangle _{Q,1} \langle g \rangle _{Q, \infty } \lvert  Q\rvert.  
\end{align*}
And from this, we see that \eqref{e:main1inft} holds.

\smallskip 

The case of \eqref{e:main11} is a little more involved,  and requires that we use all the scales in our modified sparse operator \eqref{e:til}, whereas the previous cases did not.  
Very briefly, we can dominate $ K _{\tau }$ by a  positive Calder\'on-Zygmund kernel, with 
Calder\'on-Zygmund norm at most $ \tau ^{- \frac {n-1}2- n \eta  }$. 
From this, and the known results for sparse domination of Calder\'on-Zygmund operators, the bounds \eqref{e:main11}   follow.  
To be more explicit, let $  \varphi = \mathbf 1_{\lvert  x\rvert <2 }$, 
and set $ \varphi _k (x) = 2 ^{-kn} \varphi (x 2 ^{-k})$. Then, we have 
\begin{equation*}
\lvert  K _{\tau } (x)  \mathbf 1_{\lvert  x\rvert <  \tau ^{-1 - \eta } }\rvert 
\lesssim  \tau ^{- \frac {n-1}2} \sum_{ k : 1\leq 2 ^{k} \leq  \tau ^{-1 - \eta }  } \varphi _ k (x). 
\end{equation*}
Convolution with  $ \varphi _{k}$ is an average on scale $ 2 ^{k}$, so that 
\begin{align*}
\Bigl\lvert 
\int \int K _{\tau } ( x-y) \mathbf 1_{\lvert  x-y\rvert <  \tau ^{-1 - \eta }  } 
f (y) g (x) \,dx \,dy
\Bigr\rvert
\lesssim \tau ^{-\frac {n-1}2}  \tilde \Lambda _{\tau , 1,1} (f,g).   
\end{align*}
But, the same bound holds for the remainder of the kernel $ K _{\tau }$, due to the decay estimates in \eqref{e:K}.  This completes the proof.

\end{proof}

\begin{proof}[Proof of Theorem~\ref{t:S}] 
We in fact show that for $ 0< \delta <1$ and   $ (\frac 1 r, \frac 1 s) \in \mathbf R (n, p_0, \delta )$,
\begin{equation} \label{e:YY}
\lVert  S _{\tau }  :  (r,s ,\tau )\rVert  \lesssim   \tau ^{- \delta - \eta } , \qquad 0< \tau , \eta < 1. 
\end{equation}
Above $ \delta $ is fixed, but $ \tau $ and $ \eta $ are allowed to vary. 
(The implied constant depends upon $ \eta $.) 
 This proves our Theorem, since for fixed  $ (\frac 1 r, \frac 1 s) \in \mathbf R (n, p_0, \delta )$, we have 
  $ (\frac 1 r, \frac 1 s) \in \mathbf R (n, p_0, \delta - \kappa )$, for a choice of $0< \kappa (r,s)< \delta $.  

The bounds in \eqref{e:YY} are self-dual and  can be interpolated, and so it suffices to verify the bounds above at the  vertexes $ v_1 = v _{n, \delta , 1}$ and $ v_2 = v _{n, \delta ,2 }$  of $ \mathbf R (n, p_0, \delta )$, as defined in \eqref{e:v}. 
But this is again an interpolation. To get the point $ v_1$, interpolate between the sparse bound \eqref{e:mainpc} and \eqref{e:main11}.  For the point $ v_2$, use \eqref{e:mainpc} and \eqref{e:main1inft}.  
See Figure~\ref{f:interpolate}. 

\end{proof}

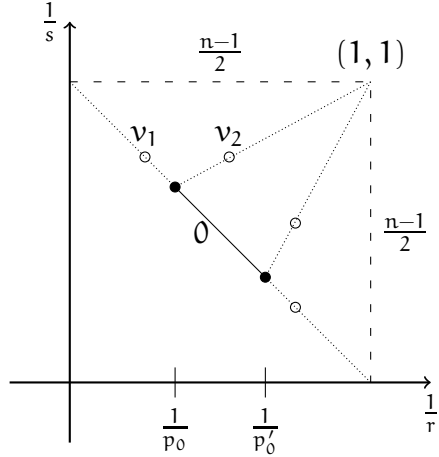
\begin{figure}
\begin{tikzpicture}[scale=4] 
\draw[thick,->] (-.2,0) -- (1.2,0) node[below] {$ \frac 1 r$};
\draw[thick,->] (0,-.2) -- (0,1.2) node[left] {$ \frac 1 s$};
\filldraw  (.35,.65) circle (.04em); \filldraw  (.65,.35) circle (.04em); 
\draw  (.35,.65) -- (.65,.35)  node[midway, left] {$ 0$}; 
\draw (.35,.05) -- (.35,-.05) node[below] {$ \tfrac 1 {p_0}$};
\draw (.65,.05) -- (.65,-.05) node[below] {$ \tfrac 1 {p_0'}$};
\foreach \position in { (1,0), (1,1) } 
    \draw[densely dotted] \position -- (.65,.35) ; 
\foreach \position in { (0,1), (1,1) } 
    \draw[densely dotted] \position -- (.35,.65) ; 
\draw[loosely dashed] (0,1) -- (1,1) node[midway, above] {$ \tfrac {n-1}2$}
node[above] {$ (1,1)$} -- (1.,0) node[midway, right] {$ \tfrac {n-1}2$}; 

\foreach \position in { (.75,.25), (.75,.53), (.25,.75), (.53,.75)} 
\draw \position circle (.04em);
\draw (.25,.75) node[above] {$ v_1$}; 
\draw (.53,.75) node[above] {$ v_2$}; 
\end{tikzpicture}

\caption{The interpolation argument for the sparse bounds. We have the sparse bounds $ (\frac 1 {p_0},  \frac1{p_0'})$, $ (1,0)$, and $ (1,1)$ as well as their duals.  The sparse bound at $ (\frac 1 {p_0},  \frac1{p_0'})$ is, up to logarithmic terms, uniformly bounded in $ 0< \tau < \frac12 $, while the others are bounded by $ \tau ^{- \frac {n-1}2}$. Interpolation, along the dotted lines, to the circles, yields sparse bounds with growth $ \tau ^{- \delta }$, for $ 0< \delta < \frac {n-1}2$.} 
\label{f:interpolate}
\end{figure}


\section{Sharpness of the sparse bounds}

We discuss sharpness of the sparse bounds in Theorem \ref{t:higher}.  Recalling that $p(\delta)$ is the critical index for the Bochner-Riesz operator $B_ \delta$, 
we cannot have any sparse bound $(r,s)$, where $1\leq r < p_ \delta$, as that would imply the boundedness of $B_\delta$ on $L^p$, for $r< p < p(\delta)$.

It remains to show sharpness of the $(r,s)$ sparse bound when 
$ p_\delta < r, s < p_\delta'$.  
This follows from a standard example. 
We work in dimensions $n\geq 2$, 
 Consider the rectangles $R$ and $\widetilde{R}$ defined by 
$$
R = \left[\tfrac {-1} \lambda , \tfrac {1} \lambda \right] \times \left[\tfrac {-c}{ \sqrt{\lambda} }, \tfrac {c}{ \sqrt{\lambda}} \right]^{n-1}; \qquad \widetilde{R} = R + \tfrac{1}{\lambda} (1,0,\ldots, 0).
$$
Above $0< c < 1$ is a small dimensional constant. 
Define the functions 
$$
f(x) = e^{i |x|} \mathbf 1_R(x), \quad  g(x) = e^{-i |x|} \mathbf 1_{\widetilde R}(x).   
$$ 
Using well known asymptotic estimates for the Bochner Riesz kernel, we have 
\begin{align}
\lvert  \langle B_{\delta} f , g \rangle \rvert   & \simeq  \Bigl\lvert  \int_{\widetilde R} \int_R \frac{e^{i (|x-y| -|x| +|y|)} } {(1+|x-y|)^{\frac{n-1}{2}+\delta}} \, dy \, dx  \Bigr\rvert
\\  \label{e:B1}
& \simeq 
\lvert R\rvert ^2 \lambda^{ \frac {n-1}2 + \delta } \simeq  \lambda^{ -\frac {n+1}2 + \delta }
\end{align}
The kernel estimates we are referencing are analogs of \eqref{e:J}, which has two exponential terms in it.
Above, one can directly verify that the phase function satisfies 
\begin{equation}
\bigl\lvert |x-y| -|x| +|y| \bigr\rvert \lesssim c, \qquad  x\in \widetilde{R},\ y\in R.  
\end{equation}
This leads to the estimate above.  
There is a  second exponential term  with phase function 
\begin{equation}
-|x-y| -|x| +|y| \simeq -2 |x-y|  , \qquad  x\in \tilde{R},\ y\in R.
\end{equation}
So, that term has substantial cancellation.

Recall that the largest value of  sparse form $\Lambda _{\mathcal S, r,s}(f,g)$  is obtained by a single sparse form. For the functions $f,g$ above, it is clear that this largest form is obtained by taking $\mathcal S$ to consist of only the smallest cube $Q$ that contains the support of both $f$ and $g$.  
 That cube has  $\ell Q \simeq \lambda^{-1}$.  And then, 
\begin{equation} \label{e:B2}
|Q| \langle f \rangle_{Q,r} \langle g \rangle_{Q,s} \simeq \lambda^{-n + \frac{n-1}{2}(\frac{1}{r} + \frac{1}{s})}.
\end{equation}

We see that the $(r,s)$ sparse bound for $B_\delta$ implies that  \eqref{e:B1} should be less than  \eqref{e:B2} for 
all  small $\lambda$.
By comparing exponents, we see that 
\begin{equation*}
- \tfrac {n+1}2 + \delta \geq 
-n + \tfrac{n-1}{2}(\tfrac{1}{r} + \tfrac{1}{s}). 
\end{equation*}
The case of equality above is the line that defines the top of the trapezoid $R(n,p(\delta),\delta)$, as is verified by inspection.


\section{Weighted Consequences} 
 
 The sparse bounds imply vector-valued and weighted inequalities for the Bochner-Riesz multipliers. 
The weights allowed are in the intersection of certain Muckenhoupt and reverse H\"older classes. 
The inequalities we can deduce are strongest at the vertex $ v _{n, \delta ,2}$, using the notation of \eqref{e:v}. 
Indeed, the weighted consequence is the strongest known for the Bochner-Riesz multipliers. 
The method of deduction follows the model of arguments in \cite{160506401}* {\S 7} and  \cite {MR3531367}*{\S 6}, as well as \cite{2017arXiv170208594L}* {\S 6}. We tread lightly around the details.  

 It is also of interest to obtain weighted bounds that more explicitly involve the Kakeya maximal function, as is done by Carbery \cite {MR768732} and Carbery and Seeger \cite{MR1765787}. We leave to the future to obtain sparse variants of these latter results.

 Recall that a weight $ w$ is in the Muckenhoupt $ A_p$ class if it has a   density $ w (dx)= w (x) dx$, with $ w (x) >0$, which is locally integrable, and $ \sigma (x) = w (x) ^{- \frac 1 {p-1}}$ is also locally integrable, and 
 \begin{equation}\label{e:Ap}
 [w] _{A_p} = \sup _{Q}  \langle w \rangle _{Q} \langle \sigma  \rangle_Q ^{p-1} < \infty , 
 \end{equation}
 where $ w (Q) = \int _{Q} w (dx)$, and the supremum is over all cubes.  We use the standard extension to $ p=1$, namely 
 \begin{equation*}
 [w] _{A_1} = \Bigl\lVert  \frac  { M w} w  \Bigr\rVert _{\infty }
 \end{equation*}
 We set $ A _{p} ^{\rho } = \{w ^{\rho } : w\in A_p\}$.  
 
 We will set $ B _{ \delta , p} $ to be the class of weights $ w$ such that we have the inequality 
 \begin{equation}\label{e:Bp}
 \lVert B _{\delta } \rVert _{L ^{p} (w) \mapsto L ^{p} (w)}. 
 \end{equation}
 Below, we will focus on qualitative results. All results can be made entirely quantitative, but given the incomplete information that we have the Bochner-Riesz conjecture, or even the full range of sparse bounds in two dimensions, we do not pursue the quantitative bounds at this time.  
 
 The best known results concerning the weighted estimates for the Bochner-Riesz multipliers in the category of $ A_p$ classes,  are below. 
 We emphasize that some of these hold for the maximal Bochner-Riesz multiplier, which we are not considering in this paper. 
 
 \begin{priorResults}\label{t:wbr} 
 \begin{enumerate}
 \item  (Christ, \cite{MR796439})  We have the inclusion below valid in all $ n\geq 2$. 
 \begin{equation}\label{e:christ}
 A _{1} ^{\frac {1+2 \delta }n} \subset B   _{\delta , 2}, \qquad   \frac {n-1} {2 (n+1)} < \delta < \frac {n-1}2. 
 \end{equation}
 
 \item (Carro, Duoandikoetxea, Lorente \cite{MR3085616}) We have the inclusion below, valid in all dimensions $ n\geq 2$.  
 \begin{equation}\label{e:cdl}
 A _{2} ^{\frac {2 \delta } {n-1} } \subset B  _{ \delta ,2}.  
 \end{equation}
 \end{enumerate}
 \end{priorResults}
 
 The second result is a consequence of the $ (\frac 12 + \frac {2 \delta } {n-1}, \frac 12 + \frac {2 \delta } {n-1})$ 
 sparse bound.  
 This latter sparse bound can be deduced from the the (trivial) $ (2,2, \tau )$ and    $(1,1, \tau )$  sparse bounds, as defined in Lemma~\ref{l:main}.  
 That is, \eqref{e:cdl} has little to do with the Bochner-Riesz operators. (The authors of \cite{MR3085616} note a similar argument.)

 We are able to deduce this improvement of \eqref{e:christ}, in that it applies for all $ 0< \delta < \frac {n-1}2$, 
 and increases the integrability of the Bochner-Riesz bound.  Finally, it approximates the known estimate at the critical index, see Theorem~\ref{t:critical}, and the earlier result of Vargas \cite{MR1405057}. 
 
 \begin{theorem}\label{t:w}
  In all dimensions $ n\geq 2$, using the notation of \eqref{e:v},   write the vertex $ v _{n, \delta ,2} = (\frac1{r}, \frac1{s })$.  We have 
 \begin{equation}\label{e:A1}
A_1 ^{ \frac{s- p (s-1)}s} \cdot   A _1 ^{1-\frac{p}r} \subset B _ {\delta ', p  }, \qquad 0< \delta < \delta ' < \tfrac{n-1}2, \ r < p < s'. 
 \end{equation}
 In particular, for $ n=2$, we have the explicit value $ v _{2, \delta , 2}=  ( \tfrac {1+6 \delta }4 , \tfrac {3 +2 \delta }4  )$, and 
 \begin{equation}\label{e:rho}
 A _{1} ^{ 1 - p \frac{1-2 \delta }4 } \cdot 
 A_1 ^{1- p \frac{1+6 \delta }4} 
 \subset B _ {\delta ', p}   , \qquad 0< \delta < \delta '< \tfrac{1}2, \  \tfrac 4{1+6 \delta }< p < \tfrac4{1-2 \delta }. 
 \end{equation}
 \end{theorem}

 This contrasts to \cite{160506401}* {Theorem 14}, which is restrictive in the range of $ 0< \delta < \frac{n-1}2$ that are allowed.  (See \cite{160506401}*{Corollary 16} for an example of the kind of vector-valued consequences that can be derived.) 
 We use the vertex $ v _{n, \delta,2 }$,  as it is the strongest sparse bound we have.  The proof is however elementary. We have this known proposition. 
 
 \begin{proposition}\label{p:one_half} If a linear operator $ T$ satisfies a $ (r , s)$ sparse bound, with $ 1\leq s < r $, we then have 
 \begin{equation*}
 \lVert T  : L ^p (w ^{\rho }) \mapsto L ^p (w ^{\rho }) \rVert  < \infty , \qquad w\in A_1 ^{ \frac{s- p (s-1)}s}   A _1 ^{1-\frac{p}r},  
 \qquad  r< p < s'. 
 \end{equation*}
 \end{proposition}
 
 \begin{proof}[Proof of Theorem \ref{t:w}] 
 It is a consequence of \cite{MR3531367}*{Prop. 6.4}, that the sparse bound assumption implies that 
 \begin{equation*}
 \lVert T  : L ^p(w  ) \mapsto L ^p  (w ) \rVert  < \infty  , \qquad r< p < s' 
 \end{equation*}
 provided the weight $ w$ is in 
 \begin{equation*}
 w \in A _{ \frac p {r}} \cap RH _{ (s'/p)'} = A _{1 } ^{ 1/ (s'/p)'} \cdot A _1 ^{1-\frac{p}r}
 \end{equation*}
 Above, $ RH _{\rho }$ denotes the reverse H\"older class of weights of index $ 1  < \rho < \infty  $, and the 
 equality above is classical.  By direct calculation, $  1/ (s'/p)'= \frac {s- p (s-1)}s$. 
 \end{proof}

\bibliographystyle{alpha,amsplain}	

\begin{bibdiv}
\begin{biblist}

\bib{MR1371114}{article}{
      author={Bak, Jong-Guk},
       title={Sharp estimates for the {B}ochner-{R}iesz operator of negative
  order in {${\bf R}^2$}},
        date={1997},
        ISSN={0002-9939},
     journal={Proc. Amer. Math. Soc.},
      volume={125},
      number={7},
       pages={1977\ndash 1986},
         url={http://dx.doi.org/10.1090/S0002-9939-97-03723-4},
      review={\MR{1371114}},
}

\bib{160506401}{article}{
      author={Benea, Cristina},
      author={Bernicot, Fr\'ed\'eric},
      author={Luque, Teresa},
       title={Sparse bilinear forms for bochner riesz multipliers and
  applications},
        date={2017},
        ISSN={2052-4986},
     journal={Transactions of the London Mathematical Society},
      volume={4},
      number={1},
       pages={110\ndash 128},
         url={http://dx.doi.org/10.1112/tlm3.12005},
}

\bib{MR3531367}{article}{
      author={Bernicot, Fr{\'e}d{\'e}ric},
      author={Frey, Dorothee},
      author={Petermichl, Stefanie},
       title={Sharp weighted norm estimates beyond {C}alder\'on-{Z}ygmund
  theory},
        date={2016},
        ISSN={2157-5045},
     journal={Anal. PDE},
      volume={9},
      number={5},
       pages={1079\ndash 1113},
  url={http://dx.doi.org.prx.library.gatech.edu/10.2140/apde.2016.9.1079},
      review={\MR{3531367}},
}

\bib{MR2860188}{article}{
      author={Bourgain, Jean},
      author={Guth, Larry},
       title={Bounds on oscillatory integral operators based on multilinear
  estimates},
        date={2011},
        ISSN={1016-443X},
     journal={Geom. Funct. Anal.},
      volume={21},
      number={6},
       pages={1239\ndash 1295},
  url={http://dx.doi.org.prx.library.gatech.edu/10.1007/s00039-011-0140-9},
      review={\MR{2860188}},
}

\bib{MR768732}{article}{
      author={Carbery, Anthony},
       title={A weighted inequality for the maximal {B}ochner-{R}iesz operator
  on {${\bf R}^2$}},
        date={1985},
        ISSN={0002-9947},
     journal={Trans. Amer. Math. Soc.},
      volume={287},
      number={2},
       pages={673\ndash 680},
      review={\MR{768732}},
}

\bib{MR1765787}{article}{
      author={Carbery, Anthony},
      author={Seeger, Andreas},
       title={Weighted inequalities for {B}ochner-{R}iesz means in the plane},
        date={2000},
        ISSN={0033-5606},
     journal={Q. J. Math.},
      volume={51},
      number={2},
       pages={155\ndash 167},
      review={\MR{1765787}},
}

\bib{MR0361607}{article}{
      author={Carleson, Lennart},
      author={Sj\"olin, Per},
       title={Oscillatory integrals and a multiplier problem for the disc},
        date={1972},
        ISSN={0039-3223},
     journal={Studia Math.},
      volume={44},
       pages={287\ndash 299. (errata insert)},
        note={Collection of articles honoring the completion by Antoni Zygmund
  of 50 years of scientific activity, III},
      review={\MR{0361607}},
}

\bib{MR3085616}{article}{
      author={Carro, Mar\'ia~J.},
      author={Duoandikoetxea, Javier},
      author={Lorente, Mar\'ia},
       title={Weighted estimates in a limited range with applications to the
  {B}ochner-{R}iesz operators},
        date={2012},
        ISSN={0022-2518},
     journal={Indiana Univ. Math. J.},
      volume={61},
      number={4},
       pages={1485\ndash 1511},
      review={\MR{3085616}},
}

\bib{MR796439}{article}{
      author={Christ, Michael},
       title={On almost everywhere convergence of {B}ochner-{R}iesz means in
  higher dimensions},
        date={1985},
        ISSN={0002-9939},
     journal={Proc. Amer. Math. Soc.},
      volume={95},
      number={1},
       pages={16\ndash 20},
      review={\MR{796439}},
}

\bib{2016arXiv161209201C}{article}{
      author={Conde-Alonso, Jos\'e~M.},
      author={Culiuc, Amalia},
      author={Di~Plinio, Francesco},
      author={Ou, Yumeng},
       title={A sparse domination principle for rough singular integrals},
        date={2017},
        ISSN={2157-5045},
     journal={Anal. PDE},
      volume={10},
      number={5},
       pages={1255\ndash 1284},
         url={http://dx.doi.org/10.2140/apde.2017.10.1255},
      review={\MR{3668591}},
}

\bib{MR544242}{article}{
      author={C\'ordoba, A.},
       title={A note on {B}ochner-{R}iesz operators},
        date={1979},
        ISSN={0012-7094},
     journal={Duke Math. J.},
      volume={46},
      number={3},
       pages={505\ndash 511},
  url={http://projecteuclid.org.prx.library.gatech.edu/euclid.dmj/1077313571},
      review={\MR{544242}},
}

\bib{MR0447949}{article}{
      author={Cordoba, Antonio},
       title={The {K}akeya maximal function and the spherical summation
  multipliers},
        date={1977},
        ISSN={0002-9327},
     journal={Amer. J. Math.},
      volume={99},
      number={1},
       pages={1\ndash 22},
      review={\MR{0447949}},
}

\bib{160305317}{article}{
      author={{Culiuc}, A.},
      author={{Di Plinio}, F.},
      author={{Ou}, Y.},
       title={{Domination of multilinear singular integrals by positive sparse
  forms}},
        date={2016-03},
     journal={ArXiv e-prints},
      eprint={1603.05317},
}

\bib{2016arXiv161208881C}{article}{
      author={{Culiuc}, A.},
      author={{Kesler}, R.},
      author={{Lacey}, M.~T.},
       title={{Sparse Bounds for the Discrete Cubic Hilbert Transform}},
        date={2016-12},
     journal={ArXiv e-prints},
      eprint={1612.08881},
}

\bib{MR1800316}{book}{
      author={Duoandikoetxea, Javier},
       title={Fourier analysis},
      series={Graduate Studies in Mathematics},
   publisher={American Mathematical Society, Providence, RI},
        date={2001},
      volume={29},
        ISBN={0-8218-2172-5},
        note={Translated and revised from the 1995 Spanish original by David
  Cruz-Uribe},
      review={\MR{1800316}},
}

\bib{MR1641626}{article}{
      author={Guti\'errez, Susana},
       title={A note on restricted weak-type estimates for {B}ochner-{R}iesz
  operators with negative index in {${\bf R}^n,\ n\ge2$}},
        date={2000},
        ISSN={0002-9939},
     journal={Proc. Amer. Math. Soc.},
      volume={128},
      number={2},
       pages={495\ndash 501},
         url={http://dx.doi.org/10.1090/S0002-9939-99-05144-8},
      review={\MR{1641626}},
}

\bib{2017arXiv170204569H}{article}{
      author={{Hyt{\"o}nen}, T.},
      author={{Petermichl}, S.},
      author={{Volberg}, A.},
       title={{The sharp square function estimate with matrix weight}},
        date={2017-02},
     journal={ArXiv e-prints},
      eprint={1702.04569},
}

\bib{170705844}{article}{
      author={{Kesler}, R.},
      author={{Lacey}, M.~T.},
       title={{Sparse Endpoint Estimates for Bochner-Riesz Multipliers on the
  Plane}},
        date={2017-07},
     journal={ArXiv e-prints},
      eprint={1707.05844},
}

\bib{160908701}{article}{
      author={{Krause}, Ben},
      author={{Lacey}, Michael~T.},
       title={{Sparse Bounds for Random Discrete Carleson Theorems}},
        date={2016-09},
     journal={ArXiv e-prints},
      eprint={1609.08701},
}

\bib{2017arXiv170208594L}{article}{
      author={{Lacey}, M.~T.},
       title={{Sparse Bounds for Spherical Maximal Functions}},
        date={2017-02},
     journal={ArXiv e-prints},
      eprint={1702.08594},
}

\bib{MR3625108}{article}{
      author={Lacey, Michael~T.},
       title={An elementary proof of the {$A_2$} bound},
        date={2017},
        ISSN={0021-2172},
     journal={Israel J. Math.},
      volume={217},
      number={1},
       pages={181\ndash 195},
  url={http://dx.doi.org.prx.library.gatech.edu/10.1007/s11856-017-1442-x},
      review={\MR{3625108}},
}

\bib{161001531}{article}{
      author={Lacey, Michael~T.},
      author={Mena~Arias, Dar\'\i~o},
       title={The sparse {T}1 theorem},
        date={2017},
        ISSN={0362-1588},
     journal={Houston J. Math.},
      volume={43},
      number={1},
       pages={111\ndash 127},
      review={\MR{3647935}},
}

\bib{160906364}{article}{
      author={{Lacey}, Michael~T.},
      author={Spencer, Scott},
       title={{Sparse Bounds for Oscillatory and Random Singular Integrals}},
        date={2016-09},
     journal={ArXiv e-prints},
      eprint={1609.06364},
}

\bib{MR2046812}{article}{
      author={Lee, Sanghyuk},
       title={Improved bounds for {B}ochner-{R}iesz and maximal
  {B}ochner-{R}iesz operators},
        date={2004},
        ISSN={0012-7094},
     journal={Duke Math. J.},
      volume={122},
      number={1},
       pages={205\ndash 232},
      review={\MR{2046812}},
}

\bib{160401334}{article}{
      author={Lerner, Andrei~K.},
      author={Ombrosi, Sheldy},
      author={Rivera-R\'\i~os, Israel~P.},
       title={On pointwise and weighted estimates for commutators of
  {C}alder\'on--{Z}ygmund operators},
        date={2017},
        ISSN={0001-8708},
     journal={Adv. Math.},
      volume={319},
       pages={153\ndash 181},
         url={http://dx.doi.org/10.1016/j.aim.2017.08.022},
      review={\MR{3695871}},
}

\bib{MR1405600}{article}{
      author={Seeger, Andreas},
       title={Endpoint inequalities for {B}ochner-{R}iesz multipliers in the
  plane},
        date={1996},
        ISSN={0030-8730},
     journal={Pacific J. Math.},
      volume={174},
      number={2},
       pages={543\ndash 553},
  url={http://projecteuclid.org.prx.library.gatech.edu/euclid.pjm/1102365183},
      review={\MR{1405600}},
}

\bib{MR1232192}{book}{
      author={Stein, Elias~M.},
       title={Harmonic analysis: real-variable methods, orthogonality, and
  oscillatory integrals},
      series={Princeton Mathematical Series},
   publisher={Princeton University Press, Princeton, NJ},
        date={1993},
      volume={43},
        ISBN={0-691-03216-5},
        note={With the assistance of Timothy S. Murphy, Monographs in Harmonic
  Analysis, III},
      review={\MR{1232192}},
}

\bib{MR1666558}{article}{
      author={Tao, Terence},
       title={The {B}ochner-{R}iesz conjecture implies the restriction
  conjecture},
        date={1999},
        ISSN={0012-7094},
     journal={Duke Math. J.},
      volume={96},
      number={2},
       pages={363\ndash 375},
      review={\MR{1666558}},
}

\bib{MR1405057}{article}{
      author={Vargas, Ana~M.},
       title={Weighted weak type {$(1,1)$} bounds for rough operators},
        date={1996},
        ISSN={0024-6107},
     journal={J. London Math. Soc. (2)},
      volume={54},
      number={2},
       pages={297\ndash 310},
         url={http://dx.doi.org.prx.library.gatech.edu/10.1112/jlms/54.2.297},
      review={\MR{1405057}},
}

\end{biblist}
\end{bibdiv}


\end{document}